\documentclass[12pt]{article}

\usepackage{amsfonts}
\usepackage{amsmath}
\usepackage{amssymb}
\usepackage{color}
\usepackage{amsthm}
\usepackage{hyperref}

\newtheorem{theorem}{Theorem}
\newtheorem{corollary}[theorem]{Corollary}

\newtheorem{lemma}[theorem]{Lemma}

\newtheorem{remark}[theorem]{Remark}

\numberwithin{theorem}{section}
\numberwithin{equation}{section}

\title{Gamma, Gaussian and Poisson approximations for random sums using size-biased and generalized zero-biased couplings} \author{Fraser Daly\footnote{Department of Actuarial Mathematics and Statistics and the Maxwell Institute for Mathematical Sciences, Heriot-Watt University, Edinburgh EH14 4AS, UK.  E-mail: f.daly@hw.ac.uk}} \date{\today}

\begin{document}

\maketitle

\noindent{\bf Abstract} 
Let $Y=X_1+\cdots+X_N$ be a sum of a random number of exchangeable random variables, where the random variable $N$ is independent of the $X_j$, and the $X_j$ are from the generalized multinomial model introduced by Tallis \cite{tallis62}. This relaxes the classical assumption that the $X_j$ are independent. We use zero-biased coupling and its generalizations to give explicit error bounds in the approximation of $Y$ by a Gaussian random variable in Wasserstein distance when either the random variables $X_j$ are centred or $N$ has a Poisson distribution. We further establish an explicit bound for the approximation of $Y$ by a gamma distribution in stop-loss distance for the special case where $N$ is Poisson. Finally, we briefly comment on analogous Poisson approximation results that make use of size-biased couplings. The special case of independent $X_j$ is given special attention throughout. As well as establishing results which extend beyond the independent setting, our bounds are shown to be competitive with known results in the independent case.   
\vspace{12pt}


\vspace{12pt}


\section{Introduction}

Let $X_1,X_2,\ldots$ be exchangeable random variables which satisfy
\begin{equation}\label{eq:model}
\mathbb{E}[e^{i(t_1X_1+\cdots+t_nX_n)}]=\rho\mathbb{E}[e^{iX_1(t_1+\cdots+t_n)}]+(1-\rho)\prod_{j=1}^n\mathbb{E}[e^{it_jX_j}]\,,
\end{equation}
for all $n\in\{1,2,\ldots\}$, all $t_j$ ($j=1,\ldots,n$), and for some parameter $\rho\in[0,1]$. This is (a slight generalization of) the generalized multinomial model introduced by Tallis \cite{tallis62}. In this model, the parameter $\rho$ interpolates between the case where the $X_j$ are independent (at $\rho=0$) and the case of comonotonicity (at $\rho=1$). 

Letting $N$ be a non-negative, integer-valued random variable independent of the $X_j$, in this note we will consider the random variable $Y=X_1+\cdots+X_N$, giving the aggregated claim amount in the collective insurance model in which $N$ represents the number of claims in a given period and the $X_j$ represent the individual claim amounts. Such random sums $Y$ also arise in a large number of other applications in various fields; see, for example, Chapter 1 of \cite{gnedenko96} for a selection of such applications.

This random sum $Y$ has been studied in the context of insurance applications by Kolev and Paiva \cite{kolev05,kolev08}, since it gives a more flexible and realistic model than the classical setting in which the $X_j$ are independent. The model (\ref{eq:model}) allows for dependence between the claim sizes $X_j$ in the form of a constant correlation coefficient. It is easily checked that $\rho$ is the correlation between $X_j$ and $X_k$ for any $j\not=k$, but we note that not every sequence of equally correlated random variables may be represented by the model (\ref{eq:model}); see Remark 3 of \cite{kolev05} for a counterexample.

Much is known about the asymptotic distribution of random sums such as $Y$ in the independent case (i.e., when $\rho=0$). We again refer the interested reader to \cite{gnedenko96}. In particular, we note that in some settings it is appropriate to consider either a Gaussian distribution, gamma distribution or Poisson distribution as a simple approximation to the distribution of $Y$, which in general is rather complicated. It is therefore natural to ask for quantitative error bounds in these approximations, and we discuss some previous work in this direction later in this introduction. It is also natural to conjecture that reasonable approximations continue to hold if the correlation parameter $\rho$ is not too large. To the best of our knowledge, this question has not been previously explored. Our primary aim  here is to establish explicit error bounds in the Gaussian, gamma and Poisson approximations for $Y$ which hold in the general case $\rho\geq0$, and which are competitive with known results in the special case where $\rho=0$. In this way we will confirm that reasonable approximations do continue to hold if $\rho$ is small enough.

We consider the approximation of $Y$ by a Gaussian random variable, a gamma random variable and a Poisson random variable in Sections \ref{sec:normal}--\ref{sec:poisson} below, respectively, where we will make use of various coupling constructions related to size-biasing, zero-biasing and some generalizations of zero-biasing defined recently by D\"obler \cite{dobler17}, which we will discuss in detail below. These constructions will then be used in conjunction with Stein's method for probability approximation to yield explicit error bounds in the approximation of $Y$. For an introduction to Stein's method for Gaussian approximation, and the use of zero- and size-biasing in conjunction with this technique, we refer the interested reader to \cite{chen11}. Poisson approximation by Stein's method is discussed in detail in \cite{barbour92}. The literature on Stein's method for gamma approximation is less well developed, but see \cite{gaunt17} and references therein for an indication of the state of the art in this area.

In line with our application to insurance models, we will quantify our approximation error in terms of the stop-loss distance, defined for real-valued random variables $Y$ and $Z$ by
\[
d_{SL}(Y,Z)=\sup_{a\in\mathbb{R}}|\mathbb{E}(Y-a)_+-\mathbb{E}(Z-a)_+|\,,
\]
where $x_+$ denotes $\max\{0,x\}$. In the Gaussian and Poisson settings, we will bound the stop-loss distance by noting that $d_{SL}(Y,Z)\leq d_{W}(Y,Z)$ for any random variables $Y$ and $Z$, where $d_W$ is the Wasserstein (or $L_1$) distance defined by
\[
d_W(Y,Z)=\sup_{h\in\mathcal{H}_W}|\mathbb{E}h(Y)-\mathbb{E}h(Z)|\,,
\]
$\mathcal{H}_W$ is the set of absolutely continuous functions $h:\mathbb{R}\to\mathbb{R}$ such that $\|h^\prime\|\leq1$, and $\|\cdot\|$ is the supremum norm given by $\|g\|=\sup_x|g(x)|$ for any real-valued function $g$.

Many authors have previously considered approximation, and especially Gaussian approximation, for such random sums $Y=X_1+\cdots+X_N$, under a variety of different assumptions on the random variables $N$ and $X_j$. To the best of our knowledge, previous work in this area applies only in the case where the $X_j$ are independent (i.e., $\rho=0$ in the model we consider), and so an important contribution of our work is to derive results which explicitly quantify the quality of the approximation in cases where $\rho>0$.   

We will highlight here some recent work that is most closely related to our own in the independent case, but refer the reader to the work of D\"obler \cite{dobler15} and Shevtsova \cite{shevtsova18_po} for more extensive discussions of the related literature. We will also restrict our attention to approximation results in Wasserstein distance, though we note that much work (including \cite{dobler15} and \cite{shevtsova18_po}) treat other probability metrics, including in particular the uniform (or $L_\infty$) distance. 

The majority of work in this area is in the direction of Gaussian approximation. Notable exceptions include \cite{shevtsova18_po}, which provides approximation results for a large number of metrics in the case where $N$ has a mixed Poisson distribution. There the author uses a conditioning argument to consider a limiting random variable of the form $Z\sqrt{\Theta}$, where $Z$ is Gaussian and $\Theta$ is a particular distributional limit for the random variables defining the Poisson mixtures, which is independent of $Z$. As one special case, Shevtsova \cite{shevtsova18_po} treats approximation by a Laplace distribution. Approximation of random sums by a Laplace distribution has also been considered by Pike and Ren \cite{pike14} using Stein's method. This technique has also been employed to approximate random sums by an exponential distribution by Pek\"oz and R\"ollin \cite{pekoz11}.

Most papers that deal with Gaussian approximation in Wasserstein distance for random sums $Y$ do so under the assumption that $N$ follows a particular distribution: the Poisson (see \cite{shevtsova18_po} and many references therein), binomial \cite{shevtsova18_bin,sunklodas14}, negative binomial \cite{shevtsova18_po,sunklodas15} or mixed Poisson \cite{shevtsova18_po} distributions, among others, have been popular choices, for reasons of both their analytic tractability and utility in applications. Exceptions include \cite{dobler15} and \cite{sunklodas13}, which give Gaussian approximation results for $Y$ in the Wasserstein distance with minimal assumptions on $N$. We note, however, that the results of \cite{sunklodas13} are often not sharp. The approach of D\"obler \cite{dobler15} is based on size biasing and Stein's method, and is therefore closest to our own. However, while \cite{dobler15} uses very general constructions, we will exploit structures that require more restrictive assumptions, but which yield sharper bounds than are otherwise available in cases where these assumptions hold and which may be applied beyond the case where the $X_j$ are independent.

We will assume in our Gaussian approximation results that either $\mathbb{E}[X_1]=0$ (with a general random variable $N$) or that $N$ has a Poisson distribution (with the $X_j$ not necessarily centred). In the case of gamma approximation we will again assume that $N$ is Poisson (again without the assumption that the $X_j$ are centred). In each case, these assumptions will allow us to exploit coupling constructions that would not otherwise hold. It will be clear from the work below where these assumptions are used: in each case, they allow a certain factorization of the variance of $Y$, which in turn gives a factorization of a certain characteristic function which is needed to give the representation we require. The assumption that $\mathbb{E}[X_1]=0$ can be thought of either as each individual summand being centred (by subtracting the mean), or as the entire random sum $Y$ being centred by subtracting a random variable proportional to $N$. For our Poisson approximation results we will not require any such assumptions on $N$ or $X_1$.    

Before detailing our Gaussian, gamma and Poisson approximation results in Sections \ref{sec:normal}--\ref{sec:poisson} below, respectively, we use the remainder of this section to give the definitions of zero biasing and size biasing that we will need. In using these definitions, we will assume without further mention throughout the work that follows that the first three moments of the random variables $X_1$ and $N$ (and therefore also of $Y$) exist and are finite, and that $N$ has positive mean.

\subsection{Size biasing}

For a non-negative random variable $X$ with mean $\mathbb{E}[X]>0$, the size-biased version of $X$, denoted by $X^s$, is defined by
\begin{equation}\label{eq:def_size}
\mathbb{E}g(X^s)=\frac{\mathbb{E}[Xg(X)]}{\mathbb{E}[X]}\,,
\end{equation}
for all functions $g:\mathbb{R}^+\to\mathbb{R}$ for which this expectation exists. If $X$ is supported on the non-negative integers, this is equivalent to writing $\mathbb{P}(X^s=j)=j\mathbb{P}(X=j)/\mathbb{E}[X]$ for all non-negative integers $j$.

In the context of Stein's method, size biasing appears most often in a Poisson approximation setting (see \cite{barbour92}, for example), where the fact that $X$ has a Poisson distribution if and only if $X^s$ has the same distribution as $X+1$ is exploited. Size biasing has also been employed in a Gaussian approximation setting; see \cite{chen11}. In our work, we will take advantage of an explicit construction of the size-biased version of our random sum $Y$ (see Lemma \ref{lem:size} below) in order to establish an explicit error bound in the approximation of $Y$ by a gamma distribution, and our Poisson approximation results will also make use of this same representation. 

\subsection{Zero biasing}

For a real-valued random variable $X$ with $\mathbb{E}[X]=0$ and $\mathbb{E}[X^2]>0$, the zero-biased version of $X$, denoted by $X^z$, is defined by
\begin{equation}\label{eq:def_zero}
\mathbb{E}[g^\prime(X^z)]=\frac{\mathbb{E}[Xg(X)]}{\mathbb{E}[X^2]}\,,
\end{equation}
for all Lipschitz functions $g:\mathbb{R}\to\mathbb{R}$ such that this expectation exists.

This definition was first developed for use with Stein's method for Gaussian approximation, motivated by the fact that $X$ is Gaussian if and only if $X^z$ has the same distribution as $X$. See \cite{chen11} for both historical background and numerous applications of zero biasing in Gaussian approximation.

Two generalizations of zero biasing have been developed by D\"obler \cite{dobler17}, each of which defines a transformation analogous to zero biasing which may be used when $\mathbb{E}[X]\not=0$. For a real-valued random variable $X$ with $\mathbb{E}[X^2]>0$, we define
\begin{itemize}
\item the generalized-zero-biased version of $X$, denoted by $X^{gz}$, which satisfies
\begin{equation}\label{eq:def_gen}
\mathbb{E}[g^\prime(X^{gz})]=\frac{\mathbb{E}[X(g(X)-g(0))]}{\mathbb{E}[X^2]}\,,
\end{equation}
for all Lipschitz functions $g:\mathbb{R}\to\mathbb{R}$ for which this expectation exists.
\item the non-zero-biased version of $X$, denoted by $X^{nz}$, which satisfies
\begin{equation}\label{eq:def_nz}
\mathbb{E}[g^\prime(X^{nz})]=\frac{\mathbb{E}[(X-\mathbb{E}[X])g(X)]}{\mbox{Var}(X)}\,,
\end{equation}
for all Lipschitz functions $g:\mathbb{R}\to\mathbb{R}$ for which this expectation exists.
\end{itemize}  
The existence of such random variables follows from Theorem 1 of \cite{dobler17}. See also Example 1 on page 104 of that paper for discussion of these transformations. We note that in the case that $\mathbb{E}[X]=0$, each reduces to zero biasing.

In our Gaussian approximation results in the case where the $X_j$ are centred, we will make use of a representation of $Y^z$; see Lemma \ref{lem:zero} below. In the case of Gaussian or gamma approximation when $N$ is Poisson, we will need the representation of $Y^{nz}$, the non-zero-biased version of our random sum, given in Lemma \ref{lem:nz} below. This representation will involve the generalized-zero-biased version of one of the summands. 

\section{Gaussian approximation} \label{sec:normal}

In this section we consider the approximation of our random sum by a Gaussian distribution. In Section \ref{sec:normal_zero} we study the case in which $\mathbb{E}[X_1]=0$, but allowing a general random variable $N$. Section \ref{sec:normal_poisson} deals with the case where $N$ is Poisson, but in which we do not require the random variable $X_1$ to be centred. 

\subsection{The zero-mean case} \label{sec:normal_zero}

Throughout this section we let $Y=X_1+\cdots+X_N$, where $N$ is a non-negative, integer-valued random variable with mean $\mathbb{E}[N]>0$, and $X_1,X_2,\ldots$ are exchangeable random variables with mean $\mathbb{E}[X_1]=0$, which are independent of $N$, and which satisfy (\ref{eq:model}). We consider the approximation of $Y$ by a Gaussian random variable $Z\sim\mbox{N}(0,\mbox{Var}(Y))$. In doing so, we will need Lemma \ref{lem:zero} below. The result of this lemma in the independent case ($\rho=0$) is implicit in the proofs of Theorems 2.5 and 2.7 in \cite{dobler15} (see the discussion below (4.8) in that paper), but we emphasise again that an important contribution of this work is to allow $\rho>0$. In this setting, Lemma \ref{lem:zero} is, to the best of our knowledge, new.
\begin{lemma}\label{lem:zero}
Let $N$ be a non-negative, integer-valued random variable, and let $X_1,X_2,\ldots$ be random variables with mean zero which satisfy (\ref{eq:model}) and are independent of $N$. Let $Y=X_1+\cdots+X_N$ and $Y^{z}$ denote its zero-biased version. Let
\begin{equation}\label{eq:tau}
\tau=\frac{\rho\mathbb{E}[N^2]}{\mathbb{E}[N]+\rho\mathbb{E}[N(N-1)]}\,,
\end{equation}
and let $I_\tau$ be a Bernoulli random variable, independent of all else, with $\mathbb{P}(I_\tau=1)=1-\mathbb{P}(I_\tau=0)=\tau$.
Then
\begin{equation}\label{eq:zerolem}
Y^z\stackrel{d}{=}I_\tau(NX_1)^z+(1-I_\tau)\left(X_1^z+\sum_{j=1}^{N^s-1}X^\prime_j\right)\,,
\end{equation}
where $N^s$ is the size-biased version of $N$, the $X^\prime_j$ are IID copies of $X_1$, all random variables on the right-hand side are independent, and ``$\stackrel{d}{=}$'' denotes equality in distribution.
\end{lemma}
\begin{proof}
For any random variable $X$, let $\phi_X(t)=\mathbb{E}[e^{itX}]$ denote its characteristic function. The model (\ref{eq:model}) gives
\begin{equation}\label{eq:Ycharfn}
\phi_Y(t)=\rho\phi_{NX_1}(t)+(1-\rho)G_N(\phi_{X_1}(t))\,,
\end{equation}
where $G_N(z)=\mathbb{E}[z^N]$ is the probability generating function of $N$. We note that, using the definition (\ref{eq:def_size}),
\begin{equation}\label{eq:sizebiasPGF}
G_{N^s-1}(z)=\frac{\mathbb{E}[Nz^{N-1}]}{\mathbb{E}[N]}=\frac{G^\prime_N(z)}{\mathbb{E}[N]}\,.
\end{equation}
Hence, combining these with the definition (\ref{eq:def_zero}) of zero biasing, we have
\begin{align}
\nonumber\phi_{Y^z}(t)&=\frac{\mathbb{E}[Ye^{itY}]}{it\mbox{Var}(Y)}=\frac{\phi^\prime_Y(t)}{t\mbox{Var}(Y)}=\frac{\rho\phi^\prime_{NX_1}(t)}{t\mbox{Var}(Y)}+\frac{(1-\rho)\phi^\prime_{X_1}(t)G^\prime_N(\phi_{X_1}(t))}{t\mbox{Var}(Y)}\\
\label{eq:lem2.1}&=\frac{\rho\mbox{Var}(NX_1)}{\mbox{Var}(Y)}\phi_{(NX_1)^z}(t)+\frac{(1-\rho)\mbox{Var}(X_1)\mathbb{E}[N]}{\mbox{Var}(Y)}\phi_{X_1^z}(t)G_{N^s-1}(\phi_{X_1}(t))\,.
\end{align}
Straightforward calculations show that $\mbox{Var}(NX_1)=\mbox{Var}(X_1)\mathbb{E}[N^2]$ and, since $\mbox{Cov}(X_j,X_k)=\rho\mbox{Var}(X_1)$ for all $k\not=j$, that
\begin{equation}\label{eq:zerovar}
\mbox{Var}(Y)=\mbox{Var}(X_1)\left\{\mathbb{E}[N]+\rho\mathbb{E}[N(N-1)]\right\}\,.
\end{equation}
Combining these with (\ref{eq:lem2.1}), we have that
\[
\phi_{Y^z}(t)=\tau\phi_{(NX_1)^z}(t)+(1-\tau)\phi_{X_1^z}(t)G_{N^s-1}(\phi_{X_1}(t))\,,
\]
which is the characteristic function of the right-hand side of (\ref{eq:zerolem}). This completes the proof.
\end{proof}

To establish a Gaussian approximation result for $Y$, we may now apply Theorem 4.1 of Chen \emph{et al.} \cite{chen11}, which was established using Stein's method. Letting 
\[
\widetilde{Y}=\frac{Y}{\sqrt{\mbox{Var}(Y)}}=\frac{Y}{\sqrt{\mbox{Var}(X_1)\left\{\mathbb{E}[N]+\rho\mathbb{E}[N(N-1)]\right\}}}
\] 
denote the standardised version of $Y$, and similarly letting $\widetilde{Z}\sim\mbox{N}(0,1)$, Theorem 4.1 of \cite{chen11} gives
\begin{equation}\label{eq:gaussianbd}
d_W(\widetilde{Y},\widetilde{Z})\leq2\mathbb{E}|\widetilde{Y}^z-\widetilde{Y}|\,,
\end{equation}
for any coupling of $\widetilde{Y}^z$ and $\widetilde{Y}$. Combining Lemma \ref{lem:zero} with the fact that $(aX)^z\stackrel{d}{=}aX^z$ for any random variable $X$ and $a\not=0$ (see Equation (2.59) of \cite{chen11}), we have that
\begin{align*}
\mathbb{E}|\widetilde{Y}^z-\widetilde{Y}|&=\frac{1}{\sqrt{\mbox{Var}(Y)}}\mathbb{E}\left|I_\tau(NX_1)^z+(1-I_\tau)\left(X_1^z+\sum_{j=1}^{N^s-1}X^\prime_j\right)-\sum_{k=1}^NX_k\right|\\
&=\frac{1}{\sqrt{\mbox{Var}(Y)}}\left(\tau\mathbb{E}|(NX_1)^z-Y|+(1-\tau)\mathbb{E}\left|X_1^z+\sum_{j=1}^{N^s-1}X_j^\prime-\sum_{k=1}^NX_k\right|\right)\\
&\leq\frac{\tau}{\sqrt{\mbox{Var}(Y)}}\left(\mathbb{E}|(NX_1)^z|+\mathbb{E}|Y|\right)+\frac{1-\tau}{\sqrt{\mbox{Var}(Y)}}\left(\mathbb{E}|X_1^z|+\mathbb{E}\left|\sum_{j=1}^{N^s-1}X_j^\prime-\sum_{k=1}^NX_k\right|\right)\,,
\end{align*}
for any coupling of $N^s$ and $N$.

Using the definition (\ref{eq:def_zero}),
\[
\mathbb{E}|X_1^z|=\frac{\mathbb{E}|X_1|^3}{2\mbox{Var}(X_1)}\,,
\]
and similarly for $\mathbb{E}|(NX_1)^z|$. Then, bounding $\mathbb{E}|Y|\leq\sqrt{\mbox{Var}(Y)}$, we have proved the following.
\begin{theorem}\label{thm:normal}
Let $N$ be a non-negative, integer-valued random variable, and let $X_1,X_2,\ldots$ be random variables with mean zero which satisfy (\ref{eq:model}) and are independent of $N$. Let $Y=X_1+\cdots+X_N$. Then
\begin{multline*}
d_W(\widetilde{Y},\widetilde{Z})\leq2\tau\left(1+\frac{\mathbb{E}[N^3]\mathbb{E}|X_1|^3}{2\mathbb{E}[N^2]\mbox{Var}(X_1)\sqrt{\mbox{Var}(Y)}}\right)\\
+\frac{2(1-\tau)}{\sqrt{\mbox{Var}(Y)}}\left(\frac{\mathbb{E}|X_1|^3}{2\mbox{Var}(X_1)}+\mathbb{E}\left|\sum_{j=1}^{N^s-1}X_j^\prime-\sum_{k=1}^NX_k\right|\right)\,,
\end{multline*}
for any coupling of $N^s$ and $N$, where $\widetilde{Z}\sim\mbox{N}(0,1)$, $\tau$ is defined by (\ref{eq:tau}) and $\mbox{Var}(Y)$ is given by (\ref{eq:zerovar}).
\end{theorem} 
This gives an explicit bound in the Gaussian approximation for $Y$. This is valid for all $\rho\in[0,1]$, though we note that we would expect a good approximation only when $\rho$ is small, as reflected in the first term of the upper bound of Theorem \ref{thm:normal}. To get some insight into the final term of the upper bound, we consider now the special case where $\rho=0$, in which case we can compare the upper bound we obtain with others available in the literature in some simple illustrative examples.

\subsubsection{The independent case}

We consider now the special case where $\rho=0$; that is, the random variables $X_1,X_2,\ldots$ are independent. In this setting, $\tau=0$, $\mbox{Var}(Y)=\mathbb{E}[N]\mbox{Var}(X_1)$, and we may choose $X_j^\prime=X_j$ for all $j$. Our upper bound therefore simplifies to that given in the following corollary.
\begin{corollary}\label{cor:normal}
Let $N$ be a non-negative, integer-valued random variable, and let $X,X_1,X_2,\ldots$ be IID random variables, independent of $N$, and with zero mean. Let $Y=X_1+\cdots+X_N$. Then
\[
d_W(\widetilde{Y},\widetilde{Z})\leq\frac{2}{\sqrt{\mathbb{E}[N]\mbox{Var}(X)}}\left(\frac{\mathbb{E}|X|^3}{2\mbox{Var}(X)}+\mathbb{E}|N+1-N^s|\mathbb{E}|X|\right)\,,
\]
for any coupling of $N^s$ and $N$, where $\widetilde{Z}\sim\mbox{N}(0,1)$.
\end{corollary}
To illustrate the upper bound of Corollary \ref{cor:normal}, we consider its application to various random variables $N$. In the following simple examples we can calculate our upper bound explicitly, and compare it with others available for these particular cases in the literature.
\begin{enumerate}
\item Let $N\sim{Po}(\lambda)$ have a Poisson distribution with mean $\lambda>0$. In this case, it is well-known that $N^s\stackrel{d}{=}N+1$ (see Chapter 1 of \cite{barbour92}), and so the bound of Corollary \ref{cor:normal} reduces to
\begin{equation}\label{eq:poeg}
d_W(\widetilde{Y},\widetilde{Z})\leq\frac{\mathbb{E}|X|^3}{\sqrt{\lambda}\mbox{Var}(X)^{3/2}}\,,
\end{equation}
which matches the bound given by Theorem 2 of \cite{shevtsova18_po}.
\item Let $N\sim\mbox{Bin}(n,p)$ have a binomial distribution. In this case, it can be easily checked using (\ref{eq:def_size}) that $N^s-1\sim\mbox{Bin}(n-1,p)$. It follows that $N+1$ is stochastically larger than $N^s$, and so
\begin{equation}\label{eq:bineg}
\mathbb{E}|N+1-N^s|=\mathbb{E}[N+1-N^s]=p\,,
\end{equation}
where we note from (\ref{eq:def_size}) that $\mathbb{E}[N^s]=\mathbb{E}[N^2]/\mathbb{E}[N]$. This gives the bound
\[
d_W(\widetilde{Y},\widetilde{Z})\leq\frac{2}{\sqrt{np\mbox{Var}(X)}}\left(\frac{\mathbb{E}|X|^3}{2\mbox{Var}(X)}+p\mathbb{E}|X|\right)
\]
from Corollary \ref{cor:normal}. This is not as sharp a bound as that given by \cite{shevtsova18_bin}, but outperforms the results of \cite{dobler15} (in terms of the constant in the upper bound) when specialized to the binomial case.
\item Similar comments apply if $N$ has a hypergeometric distribution. It is again the case that $N+1$ is stochastically larger than $N^s$ here (see Section 6.1 of \cite{barbour92}), and so $\mathbb{E}|N+1-N^s|$ can be easily evaluated in terms of the first two moments of $N$. Few authors have explicitly considered such approximation results for hypergeometric sums; one exception is D\"obler \cite{dobler15}, who evaluates his general results in the hypergeometric case. As in the binomial case above, our results outperform those of \cite{dobler15} here.
\item Suppose that $N$ has a mixed Poisson distribution, $N\sim\mbox{Po}(\Lambda)$ for some positive random variable $\Lambda$. In this case it can be easily checked that $N^s-1\sim\mbox{Po}(\Lambda^s)$. Since $\Lambda^s$ is stochastically larger than $\Lambda$, it follows from Theorem 1.A.6 of \cite{shaked07} that $N^s$ is stochastically larger than $N+1$ in this case. Hence,
\begin{equation}\label{eq:mpeg}
\mathbb{E}|N+1-N^s|=\mathbb{E}[N^s-N-1]=\frac{\mbox{Var}(\Lambda)}{\mathbb{E}[\Lambda]}\,.
\end{equation} 
Our Corollary \ref{cor:normal} thus gives
\[
d_W(\widetilde{Y},\widetilde{Z})\leq\frac{2}{\sqrt{\mathbb{E}[\Lambda]\mbox{Var}(X)}}\left(\frac{\mathbb{E}|X|^3}{2\mbox{Var}(X)}+\frac{\mbox{Var}(\Lambda)}{\mathbb{E}[\Lambda]}\mathbb{E}|X|\right).
\]
Shevtsova \cite{shevtsova18_po} also considers approximation of such mixed Poisson random sums, though (as noted above) allows for a possibly non-Gaussian limit. To compare our results with those of \cite{shevtsova18_po}, consider the special case where $\Lambda$ has a gamma distribution with mean $\frac{r(1-p)}{p}$ and variance $\frac{r(1-p)^2}{p^2}$ for some $r>0$ and $p\in(0,1)$, so that $N$ has a negative binomial distribution. For ease of comparison, we will also assume that $\mathbb{E}[X^2]=1$ here. In this case, our bound becomes
\begin{equation}\label{eq:mp_me}
d_W(\widetilde{Y},\widetilde{Z})\leq\frac{1}{\sqrt{r}}\sqrt{\frac{p}{1-p}}\mathbb{E}|X|^3+\frac{2}{\sqrt{r}}\sqrt{\frac{1-p}{p}}\mathbb{E}|X|\,,
\end{equation}
which gives the expected rate $O(r^{-1/2})$ when $p$ is fixed and $r\to\infty$; see also \cite{sunklodas15}, where a bound of the same order but with a considerably larger constant was obtained. On the other hand, Corollary 1 of \cite{shevtsova18_po} gives
\begin{equation}\label{eq:mp_sh}
d_W(\widetilde{Y},\widetilde{Z})\leq\frac{1}{\sqrt{r}}\sqrt{\frac{p}{1-p}}\mathbb{E}|X|^3+\frac{1.0801}{r}\,.
\end{equation}
Again, this gives a bound of order $O(r^{-1/2})$ when $p$ is fixed, but we note that, unlike (\ref{eq:mp_me}), the rate can improve if $p$ varies with $r$. For example, the bound of (\ref{eq:mp_sh}) is of order $O(r^{-1})$ if $p$ is of order $O(r^{-1})$. Finally, we also note that specialising the general results of \cite{shevtsova18_po} to particular cases involves evaluating an integral including the distribution function of $\Lambda$, while in our case we only need to know the first two moments of $\Lambda$, which can be considerably easier to obtain. 
\end{enumerate}
As can be noted from the examples above, the term $\mathbb{E}|N+1-N^s|$ in Corollary \ref{cor:normal} will be small when $N$ is `close to' Poisson. As an alternative approach to such approximations in the case where $N$ is close to Poisson, we could combine the bound (\ref{eq:poeg}) for Poisson sums with Lemma 2 of \cite{shevtsova18_po}, which shows that if $M$ and $N$ are non-negative, integer-valued random variables independent of $X_1,X_2,\ldots$, then
\[
d_W(X_1+\cdots+X_M,X_1+\cdots+X_N)\leq d_W(M,N)\mathbb{E}|X|\,.
\]
Choosing $M\sim\mbox{Po}(\mathbb{E}[N])$ and using the triangle inequality for Wasserstein distance then gives
\[
d_W(\widetilde{Y},\widetilde{Z})\leq\frac{1}{\sqrt{\mbox{Var}(X)\mathbb{E}[N]}}\left(\frac{\mathbb{E}|X|^3}{\mbox{Var}(X)}+d_W(M,N)\mathbb{E}|X|\right)\,.
\]
However, this can be significantly worse than the bound in our Corollary \ref{cor:normal}. For example, if $N\sim\mbox{Bin}(n,p)$ has a binomial distribution, then $\mathbb{E}|N+1-N^s|=p$ but $d_W(N,M)$ is of order $O(p\sqrt{np})$; see page 16 of \cite{barbour92}.

\subsection{The Poisson case} \label{sec:normal_poisson}

Throughout this section we let $Y=X_1+\cdots+X_N$, where $N\sim\mbox{Po}(\lambda)$ has a Poisson distribution with mean $\lambda>0$, and $X_1,X_2,\ldots$ satisfy (\ref{eq:model}). We no longer require the assumption that $X_1$ has zero mean. We will again derive an upper bound in the approximation of $Y$ by a Gaussian random variable $Z\sim\mbox{N}(\mathbb{E}[Y],\mbox{Var}(Y))$. The techniques we use will be similar to those employed in Section \ref{sec:normal_zero}, though with two key differences. Firstly, we shall need an analogue of Lemma \ref{lem:zero} for the non-zero-biased version of $Y$, and we will need to make use of results from Stein's method for Gaussian approximation with a non-zero mean. 

We begin with a representation for $Y^{nz}$, analogous to the representation for $Y^z$ given in Lemma \ref{lem:zero}.
\begin{lemma}\label{lem:nz}
Let $N\sim\mbox{Po}(\lambda)$ for $\lambda>0$, and let $X_1,X_2,\ldots$ be random variables which satisfy (\ref{eq:model}) and are independent of $N$. Let $Y=X_1+\cdots+X_N$ and $Y^{nz}$ denote its non-zero-biased version. Let
\begin{equation}\label{eq:sigma}
\sigma=\frac{\rho(\mathbb{E}[X_1^2]+\lambda\mbox{Var}(X_1))}{\mathbb{E}[X_1^2]+\lambda\rho\mbox{Var}(X_1)}\,,
\end{equation}
and let $I_\sigma$ be a Bernoulli random variable, independent of all else, with $\mathbb{P}(I_\sigma=1)=1-\mathbb{P}(I_\sigma=0)=\sigma$. Then
\[
Y^{nz}\stackrel{d}{=}I_\sigma(NX_1)^{nz}+(1-I_\sigma)\left(X_1^{gz}+\sum_{j=1}^NX_j^\prime\right)\,,
\]
where $X_1^{gz}$ is the generalized-zero-biased version of $X_1$, the $X_j^\prime$ are IID copies of $X_1$, and the random variables on the right-hand side are independent.
\end{lemma}
\begin{proof}
We proceed similarly to the proof of Lemma \ref{lem:zero}. Firstly, straightforward calculations show that $\mathbb{E}[Y]=\lambda\mathbb{E}[X_1]$ and
\begin{equation}\label{eq:poisvar}
\mbox{Var}(Y)=\lambda\mathbb{E}[X_1^2]+\lambda^2\rho\mbox{Var}(X_1)\,.
\end{equation}
Now, we use the definition (\ref{eq:def_nz}) of $Y^{nz}$ to write
\[
\phi_{Y^{nz}}(t)=\frac{\mathbb{E}[Ye^{itY}]}{it\mbox{Var}(Y)}-\frac{\mathbb{E}[Y]\mathbb{E}[e^{itY}]}{it\mbox{Var}(Y)}=-\frac{\phi_Y^\prime(t)}{t\mbox{Var}(Y)}-\frac{\mathbb{E}[Y]\phi_Y(t)}{it\mbox{Var}(Y)}\,.
\]
Applying this with the representation (\ref{eq:Ycharfn}) of $\phi_Y$, we obtain
\begin{multline}\label{eq:nzrep}
\phi_{Y^{nz}}(t)=\rho\left(-\frac{\phi_{NX_1}^\prime(t)}{t\mbox{Var}(Y)}-\frac{\mathbb{E}[Y]\phi_{NX_1}(t)}{it\mbox{Var}(Y)}\right)\\
+(1-\rho)\left(-\frac{\phi_{X_1}^\prime(t)G_N^\prime(\phi_{X_1}(t))}{t\mbox{Var}(Y)}-\frac{\mathbb{E}[Y]G_N(\phi_{X_1}(t))}{it\mbox{Var}(Y)}\right)\,.
\end{multline}
Again using the definition (\ref{eq:def_nz}), the first term on the right-hand side of (\ref{eq:nzrep}) becomes
\[
\frac{\rho\mbox{Var}(NX_1)}{\mbox{Var}(Y)}\phi_{{(NX_1)}^{nz}}(t)=\sigma\phi_{{NX_1}^{nz}}(t)\,.
\]
Since $N\sim\mbox{Po}(\lambda)$, we have $G_N(z)=e^{\lambda(z-1)}$ and $G^\prime_N(z)=\lambda G_N(z)$, so the final term on the right-hand side of (\ref{eq:nzrep}) becomes
\begin{align*}
(1-\rho)\left(-\frac{\lambda\phi_{X_1}^\prime(t)}{t\mbox{Var}(Y)}-\frac{\mathbb{E}[Y]}{it\mbox{Var}(Y)}\right)G_N(\phi_{X_1}(t))&=\frac{\lambda(1-\rho)\mathbb{E}[X_1^2]}{\mbox{Var}(Y)}\phi_{X_1^{gz}}(t)G_N(\phi_{X_1}(t))\\
&=(1-\sigma)\phi_{X_1^{gz}}(t)G_N(\phi_{X_1}(t))\,,
\end{align*}
where we use the definition (\ref{eq:def_gen}) of $X_1^{gz}$. The required result follows.
\end{proof}
\begin{remark}
\emph{We note that Lemmas \ref{lem:zero} and \ref{lem:nz} are consistent, in that they both give the same result when specialised to the case where $N\sim\mbox{Po}(\lambda)$ and $\mathbb{E}[X_1]=0$. Consider, for simplicity, the case $\rho=0$, where we can let $X_j^\prime=X_j$ for all $j$. Comparing these two lemmas, it is natural to conjecture that a representation such as $Y^{nz}\stackrel{d}{=}X_1^{gz}+\sum_{j=1}^{N^s-1}X_j$ holds in general. We have, however, been unable to prove such a result, which would include both Lemmas \ref{lem:zero} and \ref{lem:nz} (when $\rho=0$) as special cases. The proofs of Lemmas \ref{lem:zero} and \ref{lem:nz} both rely on factorisations of $\mbox{Var}(Y)$ into a product of moments of $N$ and moments of $X_1$. Such a factorisation does not hold in the general case, and we have been unable to overcome the need for such a factorisation in the proofs.}
\end{remark}
We now use Lemma \ref{lem:nz} to derive a Gaussian approximation result for $Y$, again using Stein's method. In this we cannot employ (\ref{eq:gaussianbd}) directly, but may follow the proof of that bound given in \cite{chen11} to derive an analogous result for approximation by a non-standard Gaussian distribution.

For a given function $h:\mathbb{R}\to\mathbb{R}$, we let $f:\mathbb{R}\to\mathbb{R}$ be the solution to
\begin{equation}\label{eq:gausssteineq_nonzero}
h(x)-\mathbb{E}h(Z)=\mbox{Var}(Y)f^\prime(x)-(x-\mathbb{E}[Y])f(x)\,,
\end{equation}
where we recall that $Z\sim\mbox{N}(\mathbb{E}[Y],\mbox{Var}(Y))$. Following the proof of the final inequality in Lemma 2.4 of \cite{chen11}, it is straightforward to show that if $h$ is absolutely continuous, then $\|f^{\prime\prime}\|\leq2[\mbox{Var}(Y)]^{-1}\|h^\prime\|$. Replacing $x$ by the random variable $Y$ in (\ref{eq:gausssteineq_nonzero}), taking expectations, using (\ref{eq:def_nz}), and applying Lemma \ref{lem:nz}, we may then write, for $h\in\mathcal{H}_W$,
\begin{align}
\nonumber|\mathbb{E}h(Y)-\mathbb{E}h(Z)|&=\mbox{Var}(Y)|\mathbb{E}f^\prime(Y)-\mathbb{E}f^\prime(Y^{nz})|\leq2\mathbb{E}|Y^{nz}-Y|\\
\nonumber&=2\mathbb{E}\left|I_\sigma(NX_1)^{nz}+(1-I_\sigma)\left(X_1^{gz}+\sum_{j=1}^NX_j^\prime\right)-\sum_{k=1}^NX_k\right|\\
\nonumber&=2\sigma\mathbb{E}\left|(NX_1)^{nz}-\sum_{k=1}^NX_k\right|+2(1-\sigma)\mathbb{E}\left|X_1^{gz}+\sum_{j=1}^N(X_j^\prime-X_j)\right|\\
\label{eq:gaussexp}&\leq2\sigma\left(\mathbb{E}|(NX_1)^{nz}|+\mathbb{E}|Y|\right)+2(1-\sigma)\mathbb{E}|X_1^{gz}|+2\lambda(1-\sigma)\mathbb{E}|X_1^\prime-X_1|\,.
\end{align}
We may bound 
\[
\mathbb{E}|Y|\leq\sqrt{\mathbb{E}[Y^2]}=\sqrt{\lambda(1+\lambda\rho)\mathbb{E}[X_1^2]+\lambda^2(1-\rho)\mathbb{E}[X_1]^2}\,.
\]
From the definition (\ref{eq:def_gen}), we have that
\[
\mathbb{E}|X_1^{gz}|=\frac{\mathbb{E}|X_1|^3}{2\mathbb{E}[X_1^2]}\,.
\]
We can use the definition (\ref{eq:def_nz}) to compute 
\begin{align*}
\mathbb{E}|(NX_1)^{nz}|&=\frac{\mathbb{E}[N^3]\mathbb{E}|X_1|^3-\mathbb{E}[N]\mathbb{E}[N^2]\mathbb{E}[X_1]\mathbb{E}[X_1^2\mbox{sgn}(X_1)]}{2\mbox{Var}(NX_1)}\\
&=\frac{(\lambda^2+3\lambda+1)\mathbb{E}|X_1|^3-\lambda(\lambda+1)\mathbb{E}[X_1]\mathbb{E}[X_1^2\mbox{sgn}(X_1)]}{2\left(\mathbb{E}[X_1^2]+\lambda\mbox{Var}(X_1)\right)}\,,
\end{align*}
where $\mbox{sgn}(x)=1$ if $x\geq0$, and $\mbox{sgn}(x)=-1$ otherwise. 
 
Letting $\widetilde{Y}$ and $\widetilde{Z}$ denote the standardised versions of $Y$ and $Z$, respectively, as in Section \ref{sec:normal_zero}, we have that $d_W(Y,Z)=\sqrt{\mbox{Var}(Y)}d_W(\widetilde{Y},\widetilde{Z})$. Hence, from (\ref{eq:gaussexp}) we obtain the following result.
\begin{theorem}\label{thm:normal_poisson}
Let $N\sim\mbox{Po}(\lambda)$ for $\lambda>0$, and let $X_1,X_2,\ldots$ be random variables which satisfy (\ref{eq:model}) and are independent of $N$. Let $Y=X_1+\cdots+X_N$. Then
\begin{equation}\label{eq:normal_poisson}
d_W(\widetilde{Y},\widetilde{Z})\leq\frac{1}{{\sqrt{\mbox{Var}(Y)}}}\left(2\sigma\alpha+\frac{(1-\sigma)\mathbb{E}|X_1|^3}{\mathbb{E}[X_1^2]}+2\lambda(1-\sigma)\mathbb{E}|X_1^\prime-X_1|\right)\,,
\end{equation}
where $\widetilde{Z}\sim\mbox{N}(0,1)$, $\sigma$ is defined by (\ref{eq:sigma}), $\mbox{Var}(Y)$ is given by (\ref{eq:poisvar}), and
\begin{multline*}
\alpha=\sqrt{\lambda(1+\lambda\rho)\mathbb{E}[X_1^2]+\lambda^2(1-\rho)\mathbb{E}[X_1]^2}\\
+\frac{(\lambda^2+3\lambda+1)\mathbb{E}|X_1|^3-\lambda(\lambda+1)\mathbb{E}[X_1]\mathbb{E}[X_1^2\mbox{sgn}(X_1)]}{2\left(\mathbb{E}[X_1^2]+\lambda\mbox{Var}(X_1)\right)}\,.
\end{multline*}
\end{theorem} 
As above, the upper bound of Theorem \ref{thm:normal_poisson} simplifies considerably in the case where $\rho=0$, i.e., the case where the $X_j$ are independent. In that case, $\sigma=0$ and we may choose $X_1^\prime=X_1$, so that both the first and final terms in the upper bound (\ref{eq:normal_poisson}) vanish. This leaves us with the upper bound
\[
d_{W}(\widetilde{Y},\widetilde{Z})\leq\frac{\mathbb{E}|X_1|^3}{\sqrt{\lambda}\mathbb{E}[X_1^2]^{3/2}}\,,
\] 
which was recently established by Shevtsova using different techniques in Theorem 2 of \cite{shevtsova18_po}.
 
\section{Gamma approximation in the Poisson case} \label{sec:gamma}

Throughout this section we let $Y=X_1+\cdots+X_N$, where $N\sim\mbox{Po}(\lambda)$ has a Poisson distribution with mean $\lambda>0$, and $X_1,X_2,\ldots$ are non-negative random variables (with $\mathbb{E}[X_1]>0)$, satisfying (\ref{eq:model}), and which are independent of $N$. We consider the approximation of $Y$ by a gamma random variable $Z\sim\Gamma(r,s)$ with density function given by $\frac{s^r}{\Gamma(r)}x^{r-1}e^{-sx}$ for $x>0$, where $\Gamma(\cdot)$ denotes the gamma function, and where $r>0$ and $s>0$ are chosen such that $\mathbb{E}[Y]=\mathbb{E}[Z]=r/s$ and $\mbox{Var}(Y)=\mbox{Var}(Z)=r/s^2$. Specifically,
\begin{equation}\label{eq:params}
r=\frac{\lambda\mathbb{E}[X_1]^2}{\mathbb{E}[X_1^2]+\lambda\rho\mbox{Var}(X_1)}\,,\qquad s=\frac{\mathbb{E}[X_1]}{\mathbb{E}[X_1^2]+\lambda\rho\mbox{Var}(X_1)}\,.
\end{equation}

Our approach here is motivated by our Gaussian approximation results above. We will again use Lemma \ref{lem:nz}, and will also need an analogous representation of $Y^s$ as given in Lemma \ref{lem:size} below. Note that the result of Lemma \ref{lem:size} in the special case where $\rho=0$ has been established by Arratia \emph{et al.} \cite{arratia19} (see their Section 2.2.2). Here we extend the result to cover $\rho>0$. We will state a more general version of this lemma than we will need in this section, since this more general version (without the assumption that $N$ is Poisson) will be useful in Section \ref{sec:poisson}.  
\begin{lemma}\label{lem:size}
Let $N$ be a non-negative, integer-valued random variable, and let $X_1,X_2,\ldots$ be non-negative random variables which satisfy (\ref{eq:model}) and are independent of $N$. Let $Y=X_1+\cdots+X_N$ and $Y^{s}$ denote its size-biased version. Let $I_\rho$ be a Bernoulli random variable, independent of all else, with $\mathbb{P}(I_\rho=1)=1-\mathbb{P}(I_\rho=0)=\rho$. Then
\begin{equation}\label{eq:sblem1}
Y^s\stackrel{d}{=}I_\rho(NX_1)^s+(1-I_\rho)\left(X_1^s+\sum_{j=1}^{N^s-1}X_j^\prime\right)\,,
\end{equation}
where the $X_j^\prime$ are IID copies of $X_1$, and all random variables on the right-hand side are independent.

In particular, if $N\sim\mbox{Po}(\lambda)$ for $\lambda>0$ then
\begin{equation}\label{eq:sblem2}
Y^s\stackrel{d}{=}I_\rho(NX_1)^s+(1-I_\rho)(X_1^s+Y^\prime)\,,
\end{equation}
where $Y^\prime=X_1^\prime+\cdots+X_N^\prime$, and the random variables on the right-hand side are independent.
\end{lemma} 
\begin{proof}
Our starting point is again the representation (\ref{eq:Ycharfn}), and the definition (\ref{eq:def_size}) which gives $\phi_{Y^s}(t)=(i\mathbb{E}[Y])^{-1}\phi_Y^\prime(t)$. Combining these we get
\begin{align*}
\phi_{Y^s}(t)&=\frac{1}{i\mathbb{E}[Y]}\left(\rho\phi_{NX_1}^\prime(t)+(1-\rho)\phi^\prime_{X_1}(t)G_N^\prime(\phi_{X_1}(t))\right)\\
&=\rho\phi_{(NX_1)^s}(t)+(1-\rho)\phi_{X_1^s}(t)G_{N^s-1}(\phi_{X_1}(t))\,,
\end{align*}
where the second equality follows from (\ref{eq:sizebiasPGF}). The representation (\ref{eq:sblem1}) follows. Finally, (\ref{eq:sblem2}) follows from the fact that if $N\sim\mbox{Po}(\lambda)$ then $N^s-1\stackrel{d}{=}N$.
\end{proof}
As in the Gaussian case above, we use Stein's method to derive a gamma approximation result for $Y$. Stein's method for gamma approximation was first developed by Luk \cite{luk94}. For more recent developments, see \cite{gaunt17} and references therein. We also note that the gamma distribution is a limiting case of the variance-gamma distribution, for which Stein's method was first developed by Gaunt \cite{gaunt13}. For more recent work in this area, see \cite{gaunt20} and references therein.  

Following, for example, \cite{luk94}, for a given function $h:\mathbb{R}^+\to\mathbb{R}$, we let $f:\mathbb{R}^+\to\mathbb{R}$ be the solution to 
\begin{equation}\label{eq:gammastein}
h(x)-\mathbb{E}h(Z)=xf^\prime(x)+(r-sx)f(x)\,,
\end{equation}
where we recall that $Z\sim\Gamma(r,s)$ and that the parameters $r$ and $s$ are chosen so that the first two moments of $Z$ match those of $Y$.

Replacing $x$ with the random variable $Y$ and taking expectations in (\ref{eq:gammastein}), and then using the definitions (\ref{eq:def_size}) and (\ref{eq:def_nz}), we have
\begin{equation}\label{eq:gammastein2}
\mathbb{E}h(Y)-\mathbb{E}h(Z)=\mathbb{E}[Y]\left\{\mathbb{E}f^\prime(Y^s)-\mathbb{E}f^\prime(Y^{nz})\right\}\,.
\end{equation} 
As a final ingredient from Stein's method, we will need a bound on the function $f$. Equation (1.8) of \cite{gaunt17} gives us that, for $h:\mathbb{R}^+\to\mathbb{R}$ differentiable with $h^\prime$ absolutely continuous and $\max\{|h(x)|,|h^\prime(x)|\}<ce^{ax}$ for some $c>0$ and $a<s$, we have 
\begin{equation}\label{eq:steinfactor}
\|f^{\prime\prime}\|\leq c_r\|h^{\prime\prime}\|\,,\mbox{ where }c_r=\frac{\sqrt{2\pi}+e^{-1}}{\sqrt{r+2}}+\frac{2}{r+2}\,.
\end{equation}

Letting $h_a(x)=(x-a)_+$ for $a\in\mathbb{R}^+$, we can use (\ref{eq:gammastein2}) to write
\[
d_{SL}(Y,Z)\leq\mathbb{E}[Y]\sup_{a\in\mathbb{R}^+}|\mathbb{E}f_a^\prime(Y^s)-\mathbb{E}f_a^\prime(Y^{nz})|\,,
\]
where $f_a$ is the solution to (\ref{eq:gammastein}) for the test function $h=h_a$. Unfortunately, the functions $h_a$ are not sufficiently smooth to apply the bound (\ref{eq:steinfactor}) directly, and so we use the following smoothing lemma. For any $\varepsilon>0$, we let $h_{a,\varepsilon}(x)=\mathbb{E}h_a(x+U_\varepsilon)$, where $U_\varepsilon\sim\mbox{U}(0,\varepsilon)$ is uniformly distributed on the interval $(0,\varepsilon)$.
\begin{lemma}\label{lem:smooth}
For any $\varepsilon>0$,
\[
d_{SL}(Y,Z)\leq\varepsilon+\sup_{a\in\mathbb{R}^+}|\mathbb{E}h_{a,\varepsilon}(Y)-\mathbb{E}h_{a,\varepsilon}(Z)|\,.
\]
\end{lemma}
\begin{proof}
Using the triangle inequality,
\begin{multline*}
d_{SL}(Y,Z)\leq\sup_{a\in\mathbb{R}^+}|\mathbb{E}h_a(Y)-\mathbb{E}h_{a,\varepsilon}(Y)|\\
+\sup_{a\in\mathbb{R}^+}|\mathbb{E}h_{a,\varepsilon}(Y)-\mathbb{E}h_{a,\varepsilon}(Z)|+\sup_{a\in\mathbb{R}^+}|\mathbb{E}h_{a,\varepsilon}(Z)-\mathbb{E}h_a(Z)|\,.
\end{multline*}
The first and final terms on the right-hand side are $d_{SL}(Y,Y+U_\varepsilon)$ and $d_{SL}(Z,Z+U_\varepsilon)$, respectively. By Corollary 4 of \cite{lefevre98}, each of these is at most $\varepsilon/2$. The result follows.
\end{proof}

Now, since $h_{a,\varepsilon}^{\prime\prime}(x)=\frac{1}{\varepsilon}I(x>a-\varepsilon)$, we may apply (\ref{eq:gammastein2}) with the choice $h=h_{a,\varepsilon}$ in conjunction with (\ref{eq:steinfactor}) to get
\[
\sup_{a\in\mathbb{R}^+}|\mathbb{E}h_{a,\varepsilon}(Y)-\mathbb{E}h_{a,\varepsilon}(Z)|\leq\frac{rc_r}{\varepsilon s}\mathbb{E}|Y^s-Y^{nz}|\,.
\]
Applying Lemma \ref{lem:smooth} and choosing 
\[
\varepsilon=\sqrt{\frac{rc_r}{s}\mathbb{E}|Y^s-Y^{nz}|}\,,
\] 
we have 
\[
d_{SL}(Y,Z)\leq2\sqrt{\frac{rc_r}{s}\mathbb{E}|Y^s-Y^{nz}|}\,.
\]
It remains only to bound $\mathbb{E}|Y^s-Y^{nz}|$. Using Lemmas \ref{lem:nz} and \ref{lem:size}, 
\[
\mathbb{E}|Y^s-Y^{nz}|=\mathbb{E}|I_\rho(NX_1)^s+(1-I_\rho)(X_1^s+Y^\prime)-I_{\sigma}(NX_1)^{nz}-(1-I_\sigma)(X_1^{gz}+Y^\prime)|\,.
\]
We can easily check that $\sigma\geq\rho$, so we may couple $I_\sigma$ and $I_\rho$ such that $I_\sigma=I_\rho=1$ with probability $\rho$, $I_\sigma=I_\rho=0$ with probability $1-\sigma$, and $I_\sigma=1$, $I_\rho=0$ with probability $\sigma-\rho$. We thus have
\begin{equation}\label{eq:gammabd}
\mathbb{E}|Y^s-Y^{nz}|\leq\rho\mathbb{E}|(NX_1)^s-(NX_1)^{nz}|+(\sigma-\rho)\mathbb{E}|X_1^s+Y^\prime-(NX_1)^{nz}|+(1-\sigma)\mathbb{E}|X_1^s-X_1^{gz}|\,.
\end{equation}
We can use the definition (\ref{eq:def_size}) to obtain
\[
\mathbb{E}|(NX_1)^s|=\frac{(\lambda+1)\mathbb{E}[X_1^2]}{\mathbb{E}[X_1]}\,,
\]
and similarly we may use the definition (\ref{eq:def_nz}) together with $\mathbb{E}[N^3]=\lambda(\lambda^2+3\lambda+1)$ to calculate that $\mathbb{E}|(NX_1)^{nz}|=\beta$, where
\begin{equation}\label{eq:beta}
\beta=\frac{(\lambda^2+3\lambda+1)\mathbb{E}[X_1^3]-\lambda(\lambda+1)\mathbb{E}[X_1]\mathbb{E}[X_1^2]}{2\left(\mathbb{E}[X_1^2]+\lambda\mbox{Var}(X_1)\right)}\,.
\end{equation}
Using the triangle inequality to bound the first two expectations on the right-hand side of (\ref{eq:gammabd}), we thus have the following.
\begin{theorem}\label{thm:gamma}
Let $N\sim\mbox{Po}(\lambda)$ for $\lambda>0$, and $X_1,X_2,\ldots$ be non-negative random variables which satisfy (\ref{eq:model}) and are independent of $N$. Let $Y=X_1+\cdots+X_N$. Then
\begin{multline*}
d_{SL}(Y,Z)\\
\leq2\sqrt{\lambda c_r\mathbb{E}[X_1]\left\{\frac{\lambda\rho\mathbb{E}[X_1^2]}{\mathbb{E}[X_1]}+\sigma\left(\frac{\mathbb{E}[X_1^2]}{\mathbb{E}[X_1]}+\beta\right)+(\sigma-\rho)\lambda\mathbb{E}[X_1]+(1-\sigma)\mathbb{E}|X_1^s-X_1^{gz}|\right\}}\,,
\end{multline*}
for any coupling of $X_1^s$ and $X_1^{gz}$, where $Z\sim\Gamma(r,s)$, the parameters $r>0$ and $s>0$ are given by (\ref{eq:params}), $\sigma$ is given by (\ref{eq:sigma}), $c_r$ is given by (\ref{eq:steinfactor}), and $\beta$ is given by (\ref{eq:beta}).  
\end{theorem}
As in the Gaussian case above, our upper bound simplifies considerably in the case where $\rho=0$. In that setting we obtain the following upper bound.
\begin{corollary}\label{cor:gamma}
Let $N\sim\mbox{Po}(\lambda)$ for $\lambda>0$, and $X,X_1,X_2,\ldots$ be IID non-negative random variables independent of $N$. Let $Y=X_1+\cdots+X_N$. Then
\[
d_{SL}(Y,Z)\leq2\sqrt{\lambda\mathbb{E}[X]\left(\frac{\sqrt{2\pi}+e^{-1}}{\sqrt{r+2}}+\frac{2}{r+2}\right)\mathbb{E}|X^s-X^{gz}|}\,,
\] 
for any coupling of $X^s$ and $X^{gz}$, where $Z\sim\Gamma(r,s)$ and the parameters $r>0$ and $s>0$ are given by (\ref{eq:params}) with $\rho=0$.  
\end{corollary}
We consider a simple example to illustrate the upper bound of Corollary \ref{cor:gamma}. Following, for example, Kolev and Paiva \cite{kolev08}, we let $X\sim\mbox{Be}(p)$ have a Bernoulli distribution with $\mathbb{P}(X=1)=p>0$. This application arises in a stop-loss reinsurance contract with retention $t>0$, where $X_i$ has the form $I(\xi_i>t)$, an indicator that the reinsurer has to pay for a particular claim, and $p=\mathbb{P}(\xi_i>t)$. The random sum $Y$ then counts the total number of claims to be paid by the reinsurer. In this setting, the definitions (\ref{eq:def_size}) and (\ref{eq:def_gen}), respectively, can be easily used to check that $X^s=1$ almost surely and $X^{gz}\sim\mbox{U}(0,1)$, so that $\mathbb{E}|X^s-X^{gz}|=1/2$ and Corollary \ref{cor:gamma} gives 
\begin{equation}\label{eq:gammaeg}
d_{SL}(Y,Z)\leq\sqrt{{2\lambda p}\left(\frac{\sqrt{2\pi}+e^{-1}}{\sqrt{\lambda p+2}}+\frac{2}{\lambda p+2}\right)}\,,
\end{equation}
where $Z\sim\Gamma(\lambda p,1)$.

In line with our results in the Gaussian case, we may have hoped to obtain an upper bound of order $O(1)$, rather than the order $O((\lambda p)^{1/4})$, in (\ref{eq:gammaeg}), as $\lambda p\to\infty$. To obtain this better order, we would have needed a constant $c_r$ of order $O(r^{-1})$, rather than $O(r^{-1/2})$, in (\ref{eq:steinfactor}). No such upper bound is currently available in the literature. The interested reader is referred to Remark 2.20 of \cite{gaunt13} for a more detailed discussion of the possibility of a constant of order $O(r^{-1})$ here. We note, however, that the bound (\ref{eq:gammaeg}) is still useful in spite of this limitation. Letting $\widetilde{Y}$ and $\widetilde{Z}$ denote the standardised versions of $Y$ and $Z$ respectively, as in Section \ref{sec:normal}, (\ref{eq:gammaeg}) gives an upper bound of order $O((\lambda p)^{-1/4})$ on $d_{SL}(\widetilde{Y},\widetilde{Z})$.

\section{Poisson approximation}\label{sec:poisson} 

We conclude by using Lemma \ref{lem:size} to give Poisson approximation results for the random sum $Y=X_1+\cdots+X_N$. Throughout this section we let $N$ be a non-negative, integer-valued random variable with positive mean, and we let $X_1,X_2,\ldots$ be non-negative, integer-valued random variables satisfying (\ref{eq:model}), independent of $N$, and with $\mathbb{E}[X_1]>0$. We let $Z\sim\mbox{Po}(\mathbb{E}[Y])$.

Following well-established techniques for Poisson approximation using Stein's method (see \cite{barbour92} for a detailed account), we may write
\[
d_W(Y,Z)\leq\mathbb{E}[Y]\mathbb{E}|Y+1-Y^s|\sup_{h\in\mathcal{H}_W}\|\Delta f\|\,,
\]
where $\Delta f(x)=f(x+1)-f(x)$, and $f$ satisfies $f(0)=0$ and  
\[
h(x)-\mathbb{E}h(Z)=\mathbb{E}[Y]f(x+1)-xf(x)
\]
for $x=1,2,\ldots$, for a given function $h$. Lemma 1.1.5 of \cite{barbour92} gives us that $\sup_{h\in\mathcal{H}_W}\|\Delta f\|\leq3\mathbb{E}[Y]^{-1/2}$, and our Lemma \ref{lem:size} yields
\begin{align}
\nonumber\mathbb{E}|Y+1-Y^s|&=\mathbb{E}\left|\sum_{k=1}^NX_k+1-I_\rho(NX_1)^s-(1-I_\rho)\left(X_1^s+\sum_{j=1}^{N^s-1}X_j^\prime\right)\right|\\
\label{eq:poub}&=\rho\mathbb{E}\left|\sum_{k=1}^NX_k+1-(NX_1)^s\right|+(1-\rho)\mathbb{E}\left|\sum_{k=1}^NX_k+1-X_1^s-\sum_{j=1}^{N^s-1}X_j^\prime\right|\,.
\end{align}
For the first term on the right-hand side of (\ref{eq:poub}), we write
\begin{align*}
\mathbb{E}\left|\sum_{k=1}^NX_k+1-(NX_1)^s\right|&\leq\mathbb{E}|N+1-N^s|+\mathbb{E}\left|\sum_{j=1}^N(X_j-1)\right|+\mathbb{E}\left|N^s(1-X_1^s)\right|\\
&\leq\mathbb{E}|N+1-N^s|+\mathbb{E}[N]\mathbb{E}|X_1-1|+\frac{\mathbb{E}[N^2]}{\mathbb{E}[N]}\left(\frac{\mathbb{E}[X_1^2]}{\mathbb{E}[X_1]}-1\right)\,,
\end{align*}
where we use Equation (28) of \cite{arratia19} to write $(NX_1)^s\stackrel{d}{=}N^sX_1^s$, note that $X_1^s\geq1$ almost surely, and use (\ref{eq:def_size}) to obtain $\mathbb{E}[X^s]=\mathbb{E}[X^2]/\mathbb{E}[X]$.

For the second term on the right-hand side of (\ref{eq:poub}), we have
\begin{align*}
\mathbb{E}\left|\sum_{k=1}^NX_k+1-X_1^s-\sum_{j=1}^{N^s-1}X_j^\prime\right|
&\leq\mathbb{E}|X_1^s-1|+\mathbb{E}\left|\sum_{k=1}^NX_k-\sum_{j=1}^{N^s-1}X_j^\prime\right|\\
&=\frac{\mathbb{E}[X_1^2]}{\mathbb{E}[X_1]}-1+\mathbb{E}\left|\sum_{k=1}^NX_k-\sum_{j=1}^{N^s-1}X_j^\prime\right|\,.
\end{align*}
Combining the above ingredients establishes the following.
\begin{theorem}\label{thm:poisson}
Let $N$ be a non-negative, integer-valued random variable, and let $X_1,X_2,\ldots$ be non-negative, integer-valued random variables which satisfy (\ref{eq:model}) and are independent of $N$. Let $Y=X_1+\cdots+X_N$. Then
\begin{multline*}
d_W(Y,Z)\leq3\rho\sqrt{\mathbb{E}[N]\mathbb{E}[X_1]}\left(\mathbb{E}|N+1-N^s|+\mathbb{E}[N]\mathbb{E}|X_1-1|+\frac{\mathbb{E}[N^2]}{\mathbb{E}[N]}\left[\frac{\mathbb{E}[X_1^2]}{\mathbb{E}[X_1]}-1\right]\right)\\
+3(1-\rho)\sqrt{\mathbb{E}[N]\mathbb{E}[X_1]}\left(\frac{\mathbb{E}[X_1^2]}{\mathbb{E}[X_1]}-1+\mathbb{E}\left|\sum_{k=1}^NX_k-\sum_{j=1}^{N^s-1}X_j^\prime\right|\right)\,,
\end{multline*}
for any coupling of $N^s$ and $N$, where $Z\sim\mbox{Po}(\mathbb{E}[N]\mathbb{E}[X_1])$.
\end{theorem}
Note that the upper bound of Theorem \ref{thm:poisson} is zero if $N$ is Poisson and $X_1$ is 1 almost surely, regardless of the value of $\rho$, as we would expect.
As in the previous cases we have looked at above, the upper bound simplifies considerably in the case $\rho=0$, where we may take $X_j^\prime=X_j$ for all $j$.
\begin{corollary}\label{cor:poisson}
Let $N$ be a non-negative, integer-valued random variable, and let $X,X_1,X_2,\ldots$ be IID non-negative, integer-valued random variables independent of $N$. Let $Y=X_1+\cdots+X_N$. Then
\[
d_W(Y,Z)\leq3\sqrt{\mathbb{E}[N]\mathbb{E}[X_1]}
\left(\frac{\mathbb{E}[X^2]}{\mathbb{E}[X]}-1+\mathbb{E}|N+1-N^s|\mathbb{E}[X]\right)\,,
\]
for any coupling of $N^s$ and $N$, where $Z\sim\mbox{Po}(\mathbb{E}[N]\mathbb{E}[X])$.
\end{corollary}
To illustrate the upper bound of Corollary \ref{cor:poisson}, we consider the following two examples:
\begin{enumerate}
\item Let $N\sim\mbox{Bin}(n,p)$ have a binomial distribution. Combining Corollary \ref{cor:poisson} with (\ref{eq:bineg}) gives
\[
d_W(Y,Z)\leq3\sqrt{np\mathbb{E}[X]}\left(\frac{\mathbb{E}[X^2]}{\mathbb{E}[X]}-1+p\mathbb{E}[X]\right)\,.
\]
\item As a second and final example, suppose that $N\sim\mbox{Po}(\Lambda)$ has a mixed Poisson distribution, for some positive random variable $\Lambda$. Then using (\ref{eq:mpeg}) gives
\[
d_W(Y,Z)\leq3\sqrt{\mathbb{E}[\Lambda]\mathbb{E}[X]}\left(\frac{\mathbb{E}[X^2]}{\mathbb{E}[X]}-1+\frac{\mbox{Var}(\Lambda)}{\mathbb{E}[\Lambda]}\mathbb{E}[X]\right)\,.
\]
\end{enumerate}

We conclude this section by noting that the same techniques can be used to bound the total variation distance between $Y$ and $Z$, defined by
\[
d_{TV}(Y,Z)=\sup_{A\subseteq\{0,1,2,\ldots\}}|\mathbb{P}(Y\in A)-\mathbb{P}(Z\in A)|\,.
\] 
We can again follow the same well-established techniques for Poisson approximation using Stein's method, the only change needed is a different set of test functions $h$ in place of $\mathcal{H}_W$. We instead use the set of indicator functions of subsets of non-negative integers, and can replace the bound $\sup_h\|\Delta f\|\leq3(\mathbb{E}[Y])^{-1/2}$ used above with the bound $\sup_h\|\Delta f\|\leq(\mathbb{E}[Y])^{-1}$; see Lemma 1.1.1 of \cite{barbour92}. All other parts of the proof remain unchanged. We then obtain, for example, the following analogue of Corollary \ref{cor:poisson} for the total variation distance; an analogue of Theorem \ref{thm:poisson} may also be easily written down.
\begin{corollary}
Let $N$ be a non-negative, integer-valued random variable, and let $X,X_1,X_2,\ldots$ be IID non-negative, integer-valued random variables independent of $N$. Let $Y=X_1+\cdots+X_N$. Then
\[
d_{TV}(Y,Z)\leq\frac{\mathbb{E}[X^2]}{\mathbb{E}[X]}-1+\mathbb{E}|N+1-N^s|\mathbb{E}[X]\,,
\]
for any coupling of $N^s$ and $N$, where $Z\sim\mbox{Po}(\mathbb{E}[N]\mathbb{E}[X])$.
\end{corollary}


\begin{thebibliography}{99}

\bibitem{arratia19} R.~Arratia, L.~Goldstein and F.~Kochman (2019). Size bias for one and all. \emph{Probab. Surveys} {\bf 16}: 1--61.

\bibitem{barbour92} A.~D.~Barbour, L.~Holst and S.~Janson (1992). \emph{Poisson Approximation}. Oxford University Press, Oxford.

\bibitem{chen11} L.~H.~Y.~Chen, L.~Goldstein and Q.-M.~Shao (2011). \emph{Normal Approximation by Stein's Method}. Springer, Berlin.

\bibitem{dobler15} C.~D\"obler (2015). New Berry-Esseen and Wasserstein bounds in the CLT for non-randomly centered random sums by probabilistic methods. \emph{ALEA, Lat. Am. J. Probab. Math. Stat.} {\bf 12}(2): 863--902.

\bibitem{dobler17} C.~D\"obler (2017). Distributional transformations without orthogonality relations. \emph{J. Theor. Probab.} {\bf 30}: 85--116.

\bibitem{gaunt13} R.~E.~Gaunt (2013). Rates of convergence of variance-gamma approximations via Stein’s method. D.Phil. thesis, Univ. Oxford.

\bibitem{gaunt20} R.~E.~Gaunt (2020). Stein factors for variance-gamma approximation in the Wasserstein and Kolmogorov distances. Preprint. \texttt{arXiv:2008.06088}.

\bibitem{gaunt17} R.~E.~Gaunt, A.~M.~Pickett and G.~Reinert (2017). Chi-square approximation by Stein's method with application to Pearson's statistic. \emph{Ann. Appl. Probab.} {\bf 27}(2): 720--756.

\bibitem{gnedenko96} B.~V.~Gnedenko and V.~Yu.~Korolev (1996). \emph{Random Summation: Limit theorems and applications}. CRC Press, Boca Raton.

\bibitem{kolev05} N.~Kolev and D.~Paiva (2005). Multinomial model for random sums. \emph{Insur. Math. Econ.} {\bf 37}: 494--504.

\bibitem{kolev08} N.~Kolev and D.~Paiva (2008). Random sums of exchangeable variables and actuarial applications. \emph{Insur. Math. Econ.} {\bf 42}: 147--153.

\bibitem{lefevre98} C.~Lef\`evre and S.~Utev (1998). On order-preserving properties of probability metrics. \emph{J. Theor. Probab.} {\bf 11}(4): 907--920. 

\bibitem{luk94} H.~M.~Luk (1994). Stein’s method for the gamma distribution and related statistical applications. Ph.D. thesis, Univ. Southern California, Los Angeles.

\bibitem{pekoz11} E.~A.~Pek\"oz and A.~R\"ollin (2011). New rates for exponential approximation and the theorems of R\'enyi and Yaglom. \emph{Ann. Probab.} {\bf 39}: 587--608.

\bibitem{pike14} J.~Pike and H.~Ren, Stein’s method and the Laplace distribution, \emph{ALEA Lat. Am. J. Probab. Math. Stat.} {\bf 11}: 571--587.

\bibitem{shaked07} M.~Shaked and J.~G.~Shanthikumar (2007). \emph{Stochastic Orders}. Springer, New York.

\bibitem{shevtsova18_bin} I.~G.~Shevtsova (2018). A moment inequality with application to convergence rate estimates in the global CLT for Poisson-binomial random sums. \emph{Theory Probab. Appl.} {\bf 62}(2): 278--294.

\bibitem{shevtsova18_po} I.~G.~Shevtsova (2018). Convergence rate estimates in the global CLT for compound mixed Poisson distributions. \emph{Theory Probab. Appl.} {\bf 63}(1): 72--93.

\bibitem{sunklodas13} J.~K.~Sunklodas (2013). $L_1$ bounds for asymptotic normality of random sums of independent random variables. \emph{Lith. Math. J.} {\bf 53}(4): 438--447.

\bibitem{sunklodas14} J.~K.~Sunklodas (2014). On the normal approximation of a binomial random sum. \emph{Lith. Math. J.} {\bf 54}(3): 356--365.

\bibitem{sunklodas15} J.~K.~Sunklodas (2015). On the normal approximation of a negative binomial random sum. \emph{Lith. Math. J.} {\bf 55}(1): 150--158.

\bibitem{tallis62} G.~M.~Tallis (1962). The use of a generalized multinomial distribution in the estimation of correlation in discrete data. \emph{J. R. Stat. Soc. Ser. B Methodol.} {\bf 24}(2): 530--534.

\end{thebibliography}
\end{document}